\documentclass{amsart}
\usepackage{amssymb,latexsym}

\input xypic
\usepackage[bookmarks=true]{hyperref}       

\usepackage{amssymb,latexsym,amsmath,amscd}
\usepackage{xspace}
\usepackage{enumerate}
\usepackage{graphicx}
\usepackage{vmargin}
\usepackage{todonotes}

\DeclareMathOperator{\reg}{reg}

\reversemarginpar

\theoremstyle{plain}
\newtheorem{theorem}{Theorem}[section]
\newtheorem*{theorem*}{Theorem}
\newtheorem{proposition}[theorem]{Proposition}

\theoremstyle{definition}
\newtheorem{definition}[theorem]{Definition}

\newtheorem{example}[theorem]{Example}

\newtheorem{remark}[theorem]{Remark}

\newcommand{\enm}[1]{\ensuremath{#1}}          %

\newcommand{\cal}[1]{\mathcal{#1}}

\newcommand{\CC}{\enm{\mathbb{C}}}

\newcommand{\NN}{\enm{\mathbb{N}}}
\newcommand{\RR}{\enm{\mathbb{R}}}

\newcommand{\PP}{\enm{\mathbb{P}}}

\newcommand{\Aa}{\enm{\cal{A}}}
\newcommand{\Bb}{\enm{\cal{B}}}
\newcommand{\Cc}{\enm{\cal{C}}}

\newcommand{\Rr}{\enm{\cal{R}}}

\newcommand{\Vv}{\enm{\cal{V}}}

\renewcommand{\phi}{\varphi}
\renewcommand{\theta}{\vartheta}
\renewcommand{\epsilon}{\varepsilon}

\begin{document}

\title[partially real rank]
{Partially complex ranks for real projective varieties}
\author{E. Ballico}
\address{Dept. of Mathematics\\
 University of Trento\\
38123 Povo (TN), Italy}
\email{ballico@science.unitn.it}
\thanks{The author was partially supported by MIUR and GNSAGA of INdAM (Italy).}

\begin{abstract}
Let $X(\CC )\subset \PP^r(\CC)$ be an integral non-degenerate variety defined over $\RR$. For any $q\in \PP^r(\RR)$ we study the existence
of $S\subset X(\CC)$ with small cardinality, invariant for the complex conjugation and with $q$ contained in the real linear
space spanned by $S$. We discuss the advantages of these additive decompositions with respect to the $X(\RR)$-rank, i.e. the
rank of $q$ with respect to $X(\RR)$. We describe the case of hypersurfaces and Veronese varieties.
\end{abstract}

\maketitle

\section{Introduction}
In recent years a lot of effort is devoted to the study of secant varieties of projective varieties which are
defined over $\RR$ (\cite{abc, bb, bbo, b, bs, mm, mmsv}). As in
\cite{bb} a
\emph{label} is a pair
$(a,b)\in
\NN^2\setminus
\{(0,0)\}$. The weight of a label
$(a,b)$ is the integer
$2a+b$. To know a label it is sufficient to know its weight and one of its entries. Let $\sigma : \PP^r(\CC )\to \PP^r(\CC)$
denote the complex conjugation. Let $X(\CC)\subset \PP^r(\CC)$ be an integral and non-degenerate complex projective variety.
The pair consisting of $X(\CC)$ and the embedding $X(\CC )\hookrightarrow \PP^r(\CC)$ is defined over $\RR$ if and only if
$\sigma (X(\CC)) = X(\CC)$ (the reader may take the latter as the definition of a real embedded variety). Note that
$\PP^r(\RR) =\{x\in \PP^r(\CC)\mid \sigma (x)=x\}$ and $X(\RR) = X(\CC)\cap \PP^r(\RR) =\{x\in X(\CC)\mid \sigma (x)=x\}$. 

\begin{definition}A
finite set
$S \subset X(\CC)$, $S\ne \emptyset$, is said to have a \emph{label} (resp. to have $(a,b)$ as its label, resp. to have a
label of weight $k$) if
$\sigma (S)=S$ (resp. $\sigma (S)=S$, $b =|S\cap X(\RR )|$ and $|S|=2a+b$, resp. $\sigma (S)=S$  and
$|S|=k$). 
\end{definition}
For any finite subset $S\subset X(\CC)$ let $\langle S\rangle _{\CC}$ denote the minimal complex linear subspace
of $\PP^r(\CC)$ containing $S$. Set $\langle S\rangle _{\RR}:= \langle S\rangle _{\CC}\cap \PP^r(\RR)$. Note the $\sigma (\langle S\rangle _{\CC}) =\langle S\rangle _{\CC}$
if $\sigma (S) =S$. Thus $\dim _{\CC}\langle S\rangle _{\CC}
= \dim _{\RR}\langle S\rangle _{\RR}$ if $\sigma (S) =S$. For any integer $k>0$ the $k$-secant variety $\sigma _k(X(\CC ))$ of $X(\CC)$
is the closure in $\PP^r(\CC)$ of all linear spaces $\langle S\rangle _{\CC}$ with $S\subset X(\CC)$ and $|S|=k$. If
$X(\CC)\subset \PP^r(\CC)$ is defined over $\RR$, then the variety $\sigma _k(X(\CC ))$ is defined over $\RR$ and
$\sigma _k(X(\CC))\cap \PP^r(\RR)$ is the set $\sigma _k(X(\CC ))(\RR)$ of the real points of $\sigma _k(X(\CC ))$. Quite
often  $\sigma _k(X(\CC ))(\RR)$ is much bigger than the set (sometimes called $\sigma _k(X(\RR)$) which is the closure in
$\PP^r(\RR)$ of all points with $X(\RR)$-rank $k$ (see Observation 4 in Remark\ref{aaa1}, Example \ref{aaa2} and Remark \ref{aaa4}). To see in a
concrete example how to use labels we give the following example.

\begin{proposition}\label{aaa3}
Let $X(\CC)\subset \PP^{n+1}(\CC)$ be an integral hypersurface defined over $\RR$. Every $q\in \PP^{n+1}(\RR)$ has either
$(1,0)$ or
$(0,2)$ or $(0,1)$ as a label.
\end{proposition}

\begin{remark}\label{aaa4}
Take $X(\CC)$ as in Proposition \ref{aaa3}. The proposition shows that each $q\in \PP^{n+1}(\RR)$ has a label with weight
equal to $r_{X(\CC)}(q)$.  In general not all points of
$\PP^{r+1}(\RR)\setminus X(\RR)$ have real rank $2$, even when $X(\RR)$ is Zariski dense in $X(\CC)$ (\cite{mm, mmsv}).
Compare the very interesting geometry in
\cite{bbo, bs, mm, mmsv} with the very simple description given by Proposition \ref{aaa3}.
\end{remark}

To show that labels help to get cheap additive decompositions of `` general '' points of $\PP^r(\RR)$ (or of $\sigma _k(X(\CC
))(\RR)$), i.e. all except a part with  smaller real dimension (and in particular smaller Hausdorff dimension) we prove the
following Theorem \ref{g1}. Let $X_{\reg}(\CC)$ denote the set of smooth points of $X(\CC)$. Set $X_{\reg}(\RR):= X_{\reg}(\CC)\cap X(\RR)$. Since $X(\CC)$ is an integral variety,
$X_{\reg}(\CC)$ is a connected complex manifold and $\dim X_{\reg}(\CC) =\dim X(\CC)$. If $X_{\reg}(\RR) \ne \emptyset$, then $X_{\reg}(\RR)$ is a real analytic manifold
of real dimension $\dim X_{\reg}(\CC)$, which is Zariski dense in $X_{\reg}(\CC)$ (Remark \ref{aad}), but closed in
$X_{\reg}(\CC)$ in the euclidean topology. We recall that generic uniqueness holds for a secant variety $\sigma _k(X(\CC))$ if
for a general $q\in
\sigma _k(X(\CC))$ there is a unique set $S\subset X(\CC)$ such that $|S| =k$ and $q\in \langle S\rangle$; here ``general ''
means ``for all $q$ in a non-empty Zariski open subset of $\sigma _k(X(\CC))$ '', but to test this condition it is sufficient
to prove it for all points $q$ of a non-empty open subset of $\sigma _k(X(\CC))$ for the euclidean topology.

\begin{theorem}\label{g1}
Let $X(\CC) \subset \PP^r(\CC)$ be an integral and non-degenerate variety defined over $\RR$ and such that $X_{\reg}(\RR)\ne \emptyset$. Let $g$ be the minimal integer
such that $\sigma _g(X(\CC)) =\PP^r(\CC)$. Assume that generic uniqueness holds for $\sigma _{g-1}(X(\CC))$. Then there
exists an open subset $U\subset \PP^r(\RR)$ (for the Zariski topology) such that $ \dim \PP^r(\RR)\setminus U \le r-1$ (and
in particular $\PP^r(\RR)\setminus U$ has measure $0$ and contains no euclidean open subset) and each $q\in U$ has a
label of weight $g+1$.
\end{theorem}

The main application of Theorem \ref{g1} is when $X(\CC)$ is a Veronese embedding of $\PP^n$. Let $\nu _d: \PP^n(\CC) \to \PP^r$, $r:= -1+\binom{n+d}{n}$, be the order $d$ Veronese embedding of $\PP^n$, i.e. the embedding associated to the vector space $S^d \CC^{n+1}$ of all homogeneous degree $d$ complex polynomials in $n+1$ variables. Set $X(\CC):= \nu _d(\PP^n(\CC))$. In this case for any finite $S\subset \PP^n(\CC)$ the linear space $\langle \nu _d(S)\rangle _{\CC}$ is the $\CC$-linear span of all $\ell _p^d$, where $p\in S$ and $\ell _p$ is the linear form whose equivalence class corresponds to $p\in \PP^n(\CC)$. Both $X(\CC)$ and the embedding $X(\CC )\hookrightarrow \PP^n(\CC )$ are defined over $\RR$. If $d$ is odd and $S\subset \PP^n(\RR)$ (with $|S|$ minimal) the interpretation above  gives the definition of $X(\RR)$-rank. If $d$ is even we allow signs, i.e. to define the $X(\RR)$-rank or real rank  of $f\in  S^d\RR^{n+1}$ we allow as addenda $\pm \ell _p^d$. As an immediate consequence of Theorem \ref{g1} and taking $r: = g-1$ in \cite[Theorem 1.1]{cov1}
we get the following result.

\begin{theorem}\label{g2}
Fix integers $n\ge 1$ and $d\ge 3$. Assume $(n,d) \notin \{(2,6), (3,4), (5,3)\}$. Set $r:= -1+\binom{n+d}{n}$ and $g:= \lceil (r+1)/(n+1)\rceil$. Let $X(\CC)\subset \PP^r(\CC)$ be the order $d$ Veronese embedding. Let $\Bb$
be the set of $x\in \PP^n(\RR)$ without  a label of weight  $g+1$. Then $\Bb$ is contained in a
real hypersurface of $\PP^r(\RR)$ and in particular it has measure $0$ and contains no non-zero euclidean open subset.
\end{theorem}

We stress that our bounds do not depend on the real algebraic geometry of $X(\RR)$ or $X(\CC)$. In our opinion they give a very strong minimal way (minimal number of real parameters) to represents almost all $\PP^r(\RR)$ using finitely many charts with minimal number of parameters. The number of needed charts is upper bounded in an explicit way. Easy examples (the real rational normal curve) shows that the bound is sharp.

If $\dim X(\CC)=1$ generic uniqueness always holds over $\CC$ for any submaximal secant variety (\cite[Corollary 2.8]{cc2}) and
hence as an easy consequence of Theorem
\ref{g1} we get the following result.

\begin{proposition}\label{g3}
Let $X(\CC )\subset \PP^r(\CC)$ be an integral and non-degenerate curve defined over $\RR$ and with $X(\RR)$ infinite.
The set of all $p\in \PP^r(\RR)$ without a label of weight $\lfloor (r+5)/2\rfloor$ 
has real dimension $<r$.
\end{proposition}

We extend Theorem \ref{g1} to arbitrary secant varieties in the following way.

\begin{theorem}\label{i4}
Let $X(\CC) \subset \PP^r(\CC)$ be an integral and non-degenerate variety defined over $\RR$ and such that $X_{\reg}(\RR)\ne
\emptyset$. Fix an integer $k\ge 2$ such that generic uniqueness holds for $\sigma _{k-1}(X(\CC))$. Then
there exists a Zariski open subset $U\subset \sigma _k(X(\CC))(\RR )$ such that $ \dim _{\RR}
\sigma _k(X(\CC))(\RR ) \setminus U \le \dim _{\RR} \sigma _k(X(\CC))(\RR )-1$ and each $q\in U$ has a label of weight $k+1$.
\end{theorem}

We recall that if $\dim X(\CC) =n$ and $\sigma _{k+n-2}(X(\CC))$ has dimension $(k+n-2)(n+1)-1$, then generic uniqueness holds
for
$\sigma _{k-1}(X(\CC))$ (\cite[Theorem 5.1]{bbc}). Thus Theorem \ref{i4} may be applied in a huge number of cases (\cite{ab3,
acv, bc1, chio, cov1, mr}; see
\cite{bbc} for a partial, but longer, list). In particular it applies to several Segre embeddings of multiprojective spaces,
i.e. to  tensors, and of Segre-Veronese embeddings of a multiprojective space, i.e. to partially symmetric tensors. We
mention that if
$a>b>0$ and generic uniqueness holds for $\sigma _a(X(\CC ))$, then it holds for $\sigma _b(X(\CC ))$ (\cite[Proposition 2.3]{bbc}) and in particular
it holds for the Veronese embeddings of a projective space, except for a very short list. The drawback of Theorem
\ref{i4} is that for these examples it covers just one case not covered in \cite{bb}. Indeed if generic uniqueness holds for
$\sigma _k(X(\CC ))$ it is sufficient to use labels of weight $k$ by \cite[Theorem 3]{bb} and all of them are necessary, since different labels cover disjoint non-empty euclidean open subsets.

\begin{remark}We explain here why we think that our approach is computationally promising. We just consider $\PP^r(\RR)$, but Theorem
\ref{i4} may be used for the real parts of arbitrary $k$-secants varieties. Let
$g$ be the first positive integer such that
$\sigma _g(X(\CC)) =\PP^r(\CC)$. For each $q\in \PP^r(\RR)$ let $r_{X(\RR)}(q)$
denote the smallest cardinality of a set $S\subset X(\RR)$ such that $q\in \langle S\rangle _{\RR}$ (\cite{an, abc, bb, bbo, b,
bs, co, mmsv}). An integer
$x$ is a
\emph{typical rank}
of $X(\RR)$ if there is a non-empty open subset $\Delta _x$ for the euclidean topology such that $r_{X(\RR)}(q) = x$
for all $q\in \Delta _x$. Let $E$ be the set of all typical ranks of $X(\RR)$. Taking each $\Delta _x$ maximal (i.e. taking
the interior for the euclidean topology of the set of all points with $X(\RR)$-rank $x$) the set
$\PP^r(\RR)
\setminus
\cup _{x\in E}
\Delta _x$ has measure zero and usually the game is to handle all $q\in \cup _{x\in E} \Delta _x$. For each
$x\in E$ to describe each of $\Delta _x$ we need  a subset of $X(\RR)$ with cardinality $x$ and hence roughly speaking we need
$xn$ real parameters, where $n:= \dim _{\CC} X(\CC)$ (unless of course we know that a much smaller subset of $X(\RR)^x$ will do,
but each case needs a detailed study to restrict the subsets of $X(\RR)$ with cardinality $x$ which must be used to give
$\Delta _x$). Call $g'$ the maximal typical rank. Roughly speaking, we need $ng'$ real parameter for at least one euclidean open
subset of $\PP^r(\RR )$. The integer
$g$ is the minimal typical rank. It is known that $g'\le 2g$ and often $g'\le 2g-1$ or $g'\le 2g-2$ in some cases (\cite{bt,
bbo}). In Theorem \ref{g1} (and hence in its corollaries like Theorem \ref{g2} and Proposition \ref{g3}), we need to take labels of  weight $g+1$.
There are $\lfloor (g+1)/2\rfloor$  labels of weight $g+1$ (resp. $g$). For any label $(a,b)$ of weight
$2a+b$ we only need $(2a+b)n$ real parameters. Indeed we take $b$ distinct points $p_1,\dots ,p_b\in X(\RR)$ and
$a$ sufficiently general points $q_1,\dots ,q_a\in X(\CC)\setminus X(\RR )$. The choice of $q_1,\dots ,q_a$ depends on $2an$
real parameters. Then we add the complex conjugate of $q_1,\dots ,q_a$. On the contrary, to use typical ranks we sometimes need
almost $2gn$ real parameters (e.g. for the typical ranks of bivariate polynomial of the bivariate polynomials, a problem
settled by G. Blekherman in \cite{b}).
\end{remark}

Using the labels we may also handle real varieties $X(\CC )\subset \PP^r(\RR )$ with $X_{\reg}(\RR)=\emptyset$ (even with
$X(\RR)=\emptyset$) for which the notion of typical rank (or even the real rank if $X(\RR) =\emptyset$) is not very
interesting. We only consider labels $(a,0)$, i.e. corresponding to $a$ points of $X(\CC)\setminus X(\RR)$ and their complex
conjugates. We prove the following result.

\begin{theorem}\label{e1}
Assume $X_{\reg}(\RR )=\emptyset$. Let $g$ be the generic rank of $X(\CC)$. Fix an even integer $k$ such that $4 \le k\le g$.
Assume that generic uniqueness holds for $\sigma _{k-2}(X(\CC))$. Then $\sigma
_k(X(\CC))(\RR)$ is Zariski dense in $\sigma _k(X(\CC))$ and there is an open subset $U$ of $\sigma _k(X(\CC))(\RR)$
such that  $\sigma _k(X(\CC))(\RR)\setminus U$ is a real hypersurface of $\sigma _k(X(\CC))(\RR)$ and for each $q\in U$ there
is $S\subset X_{\reg}(\CC)$ with $\sigma (S)=S$, $|S| = k+2$ and $S$ of label $(\frac{k+2}{2},0)$.
\end{theorem}

\begin{remark}Since we allow labels with weight $>r_{X(\CC)}(q)$ to cover $q\in \PP^r(\RR)$, the open subsets covering almost all $\PP^r(\RR)$ 
in Theorems \ref{g1} and \ref{g2} and Proposition \ref{g3} (or almost all $\sigma _k(X(\CC ))(\RR)$ in Theorem \ref{i4}) may
overlap. In Theorem \ref{e1} we have a unique label, but even in this case many $q\in \PP^r(\RR)$ may be in the linear span of
different $\sigma$-invariant sets with label
$(\frac{k+2}{2},0)$.\end{remark}

\begin{remark}
As in \cite{bb} the interested reader may work over an arbitrary real closed field $\Rr$ instead of the field $\RR$, just using $\Cc := \Rr (i)$ (an algebraic closure of $\Rr$)
instead of $\CC$.\end{remark}

I thank a referee for useful suggestion on the organization of the paper.

\section{The proofs}
\begin{remark}\label{aad}
Let $X(\CC)$ be an integral projective variety defined over $\RR$. Set $n:= \dim_{\CC} X(\CC)$. The set $X(\RR)$ is a closed subset of the compact $r$-dimensional complex
space $\PP^r(\RR)$. Let $X_{\reg}(\CC)$ be the set of all smooth points of $X(\CC)$. Since $X(\CC)$ is an integral variety, $X_{\reg}(\CC)$ is a non-empty connected smooth $n$-dimensional complex manifold. Let $M(\CC)$ be a smooth and connected $n$-dimensional complex manifold defined over $\RR$ and let $M(\RR)$ denote the set
of all its real points. Assume $M(\RR) \ne \emptyset$ and fix $p\in M(\RR)$. Call $\Aa_{M(\CC),p}$ the local ring at $p$ of the
complex analytic manifold $M(\CC)$. Obviously $M(\RR)$ is closed in $M(\CC)$. It is well-known that $M(\RR)$ is a smooth (but
not necessarily connected) differential manifold of pure dimension $n$ (\cite[Ch. II, Corollary 4.11]{gmt}). Now assume that
$M(\CC)$ is the complex analytic manifold associated to a complex algebraic variety defined over $\RR$ and that the real
structure of the complex analytic manifold $M(\CC)$ is the one induced by the assumption that this algebraic variety is
defined over
$\RR$. The set $M(\RR)$ is the same in the algebraic and in the analytic category. We  claim that $M(\RR)$ is dense for the
Zariski topology of the algebraic variety $M(\CC)$. To prove the claim it is sufficient to prove that if $f \in \Aa_{M(\CC),p}$
and if
$f$ vanishes at the germ of $M(\RR)$ at $p$, then $f=0$. Suppose for instance that $M(\CC)$ (as an analytic manifold) contains
a euclidean neighborhood $U$ of $0\in \CC$ with the usual complex conjugation involution with $\RR$ and write $\Aa _{\CC
^n,0}$ for the local ring of convergent power series in $n$ complex variables. If $f\in \Aa _{\CC^n,0}$ vanishes at each point
of $U\cap \RR^n$, then all coefficients of the power series of $f$ vanishes (this is easily reduced to the case $n=1$ in which
it is obvious that if $f\ne 0$, then $f$ has at most countably many zeros). By \cite[Proposition 4.2]{gmt} the involution
$\sigma: M(\CC)\to M(\CC)$ is uniquely determined by the set of its fixed points (when non-empty).
\end{remark}

\begin{remark}\label{aaa1} 
Let $X(\CC)\subset \PP^r(\CC)$ be an integral and non-degenerate variety defined over $\RR$ (no
assumption on $X(\RR)$ and/or $X_{\reg}(\RR)$). Fix an even integer $k\ge 2$. Just to simplify the language we assume  $k\le r+1$, so that $k$ general points of $X(\CC)$ are
linearly independent.
For any positive integer $t$ set $X(\CC)^{(t)}:= \{(a_1,\dots ,a_t)\in X(\CC)^t\mid a_i\ne a_j$ for all $i\ne j\}$.  Since $X(\CC)^{(t)}$ is a non-empty Zariski open subset of $X(\CC)^t$,
it is an irreducible quasi-projective variety. It is invariant for the complex conjugation $\sigma$ and the fixed point set  is just the set $X(\RR)^t\cap X(\CC)^{(t)}$.
The symmetric group $S_t$ of permutations of $\{1,\dots ,t\}$ acts on $X(\CC)^{(t)}$ and this action commutes with the complex conjugation. Since $X(\CC)^{(t)}$ is a quasi-projective variety, there is a quasi-projective variety $X(\CC)_{(t)}$ parametrizing the orbits for this action of $S_t$ (\cite[p.
111]{mu}) and $X(\CC)_{(t)}$ is defined over $\RR$.
$X(\CC)_{(t)}$ parametrizes the subsets $S\subset X(\CC)$ with cardinality $t$. An element $S\in X(\CC)_{(t)}$ is an
element of
$X(\CC)_{(t)}(\RR)$ if and only if $\sigma (S)=S$. Thus if $q\in X(\CC)\setminus X(\RR)$, the set $\{q,\sigma (q)\}\in
X(\CC)_{(2)}(\RR)$, while $\{q,\sigma (q)\}\nsubseteq X(\RR)$. Thus for any even $t\ge 2$ there are many $S\in
X(\CC)_{(t)}(\RR)$ which are not contained in $X(\RR)$. If $X(\RR)\ne \emptyset$ this is also true for all odd integers $t\ge
3$.

\quad \emph{Claim 1:} For all even integers $t\ge 2$ the set $X(\CC)_{(t)}(\RR)$ is Zariski dense in $X(\CC)_{(t)}$.

\quad \emph{Proof of Claim 1:} Since $X(\CC)_{(t)}$ is an irreducible complex variety, to prove Claim 1 it is sufficient to prove that the Zariski closure
of $X(\CC)_{(t)}(\RR)$ in $X(\CC)_{(t)}$ contains a non-empty euclidean open subset. First assume $t=2$, take some $q\in X_{\reg}(\CC)\setminus X_{\reg}(\RR)$ and
a euclidean neighborhood $U$ of $q\in X_{\reg}(\CC)\setminus X_{\reg}(\RR)$. All points $\{q,\sigma (q)\}$, $q\in U$, are
contained in $X(\CC)_{(2)}(\RR)$ and their union give a euclidean neighborhood of $\{q,\sigma (q)\}$ in $X(\CC)_{(2)}$. Now
assume $t\ge 4$ and set $b:= t/2$. We take $\{q_1,\dots ,q_b\}\in X(\CC)_{(b)}$ with the additional restriction that $q_i \ne
\sigma (q_j)$ for any $i, j$ and use all these points $\{q_1,\dots ,q_b,\sigma (q_1),\dots ,\sigma (q_b)\in X(\CC)_{(t)}(\RR)$.

\quad \emph{Observation 1:} If $X_{\reg}(\RR)\ne \emptyset$ for all integers $t\ge 1$  the set $X(\CC)_{(t)}(\RR)$ is Zariski dense in $X(\CC)_{(t)}$.

\quad \emph{Claim 2:} The set $\sigma _k(X(\CC))(\RR)$ is Zariski dense in $\sigma _k(X(\CC))$.

\quad \emph{Proof of Claim 2:} For all integers $t\ge 1$ set 
$$X(\CC)_{\langle t\rangle}:= \{S\in X(\CC)_{(t)}\mid S \ \textit{is linearly independent}\}.$$
$X(\CC)_{\langle t\rangle)}$ is a non-empty Zariski open subset of $X(\CC)_{(t)}$ defined over $\RR$ and $X(\CC)_{\langle
t\rangle}(\RR) = X(\CC)_{\langle t\rangle}\cap X(\CC)_{(t)}(\RR)$. Thus $X(\CC)_{\langle t\rangle}(\RR)$ is Zariski dense in
$X(\CC)_{\langle t\rangle}$. For each $S\in X(\CC)_{\langle t\rangle}(\RR)$ the linear space $\langle S\rangle_{\CC}$ is
defined over $\RR$ and hence $\langle S\rangle_{\CC}\cap \PP^r(\RR)$ is a real projective space of dimension $t-1$. Since $k$
is even, Claim 1 shows that the union of these real projective spaces is Zariski dense in $\sigma _k(X(\CC))$.

\quad \emph{Observation 2:} Take $X(\CC)$ with $X(\RR) =\emptyset$. Claim 2 gives that $\sigma _2(X(\CC))(\RR)$ is large. For
instance we may take as $X(\CC)$ a smooth plane conic $X(\CC)\subset \PP^2(\CC)$ defined over $\RR$ but with no real point (the
smooth conic $x_0^2+x_1^2+x_2^2=0$). We get $\sigma _2(X(\CC)) =\PP^2(\RR)$. Proposition \ref{aaa3} shows that each $q\in
\PP^2(\RR)$ has $(1,0)$ as a label.

\quad \emph{Observation 3:} If $X_{\reg}(\RR)\ne \emptyset$, then $\sigma _x(X(\CC))(\RR)$ is Zariski dense in $\sigma
_x(X(\CC))$ also for odd $x$. In many cases with $x\ge 2$ it is larger than the union of all linear spaces $\langle S\rangle
_{\RR}$ with $S\subset X(\RR)$ and $\sharp (S) \le x$.

\quad \emph{Observation 4:} Fix an integer $x\ge 2$ and assume that each subset of $X(\CC)$ with cardinality $2x$ is linearly
independent. This is sufficient to get that each $q\in \PP^r(\CC)$ with $r_{X(\CC)}(q)=x$ is in the linear span of a unique
set $S\subset X(\CC)$ with cardinality $x$. Thus $q\in \PP^r(\RR)$ is and only if $\sigma (S)=S$. Since $x\ge 2$, many such
points
$q$ are not in the linear span of the union of $x$ points of $X(\CC)$.
\end{remark}

\begin{example}\label{aaa2}
Let $X(\CC) \subset \PP^3(\CC)$ be the rational normal curve. Let $\tau (X(\CC))$ be its tangent develobale. By Sylvester's
theorem (\cite{cs}) a point $q\in \PP^3(\CC)$ has $r_{X(\CC)}(q)=2$ if and only if $q\in \PP^3(\CC)\setminus \tau (X(\CC))$.
Fix $q\in \PP^3(\RR)\setminus \tau (X(\CC))\cap \PP^3(\RR)$ and fix $S\subset X(\CC)$ such that $q\in \langle S\rangle_{\CC}$
and $|S|=2$. Since any $4$ points of $X(\CC)$ are linearly independent. Thus $S$ is unique. Since $\sigma (q)=q$
we get $\sigma (S)=S$. Thus $S$ has a label. Call $U$ (resp. $U'$) the set of all $q\in \PP^3(\RR)\setminus \tau (X(\CC))\cap
\PP^3(\RR)$ such that $S$ has label $(2,0)$ (resp. $(0,1)$). We have $U\cap U'=\emptyset$, $U\cup U' =
\PP^3(\RR)\setminus \tau (X(\CC))\cap \PP^3(\RR)$ and $U$ and $U'$ are semialgebraic sets of real dimension $3$.
\end{example}

\begin{proof}[Proof of Proposition \ref{aaa3}:]
If $q\in X(\RR)$, then $(1,0)$ is a label for $q$. Fix $q \in
\PP^{n+1}(\RR)\setminus X(\RR)$. Let $T(\CC)$ be the set of all lines $L\subset \PP^{n+1}(\CC)$ containing $o$.
The set
$T(\CC)$ is an $n$-dimensional complex space. Since $q \in
\PP^{n+1}(\RR)$, $T(\CC)$ is the complexification of a real $n$-dimensional projective space $T(\RR)$. For any $x\in
\PP^{n+1}(\RR)\setminus \{q\}$ the lines spanned by $\{q,x\}$ is an element of $T(\RR)$. Since $q\notin X(\CC)$, no element
of $T(\CC)$ is contained in $X(\CC)$. The set of all
$L\in T(\CC)$ tangent to $X(\CC)$ is a proper closed algebraic subset $\Delta (\CC)\subseteq T(\CC)$ defined over $\RR$.
Since $X_{\reg}(\RR)\ne \emptyset$ and each element of $T(\CC)$ meets $X(\CC)$, there is $o\in X_{\reg}(\RR)$ such that
the line $L$ spanned by $\{q,o\}$ is transversal to $X(\CC)$. Set $d:= \deg (X(\CC))$ and $S:= L\cap X(\CC)$. Since $L\in
T(\RR)$, we have $\sigma (S) = S$. Since $L\notin \Delta$, we have $|S|=d$. Since $\sigma (S) =S$, there is a set $S'\subseteq S$
such that $\sigma (S')=S'$ and $|S'|=2$. The line $L$ is spanned by $S'$. The set $S'$ has either label $(1,0)$ or label
$(0,2)$.
\end{proof}

\begin{proof}[Proof of Theorem \ref{g1}:]
Set $n:= \dim X(\CC)$. Since generic uniqueness holds for the secant variety $\sigma _{g-1}(X(\CC))$, $\sigma _{g-1}(X(\CC))$ is not defective,
i.e. it has dimension $(n+1)(g-1)-1$. Let $\Vv \subset \sigma _{g-1}(X(\CC))(\RR)$ be the set of all points with a label of
weight $g-1$. Since generic uniqueness holds for $\sigma _{g-1}(X(\CC))$, the set $\sigma _{g-1}(X(\CC))\setminus \Vv$ has
real dimension $<(n+1)(g-1)$ (\cite[Theorem 3]{bb}), i.e., since it is a semialgebraic set, it does not contain a non-empty
open subset of $\PP^r(\RR)$ for the euclidean topology.  A Zariski dense subset of $\PP^r(\CC)$ is obtained taking the union
of all $\langle \{x,y\}\rangle _\CC$ with $x\in \Vv$ and
$y\in X(\CC)\setminus X(\RR)$ (Claims 1 and 2 of Remark \ref{aaa1}). Thus we get a Zariski dense subset of $\PP^r(\CC)$ taking
the union of all complex linear spaces $\langle \{x,y,\sigma (y)\}\rangle _\CC$. The complex linear space $\langle \{x,y,\sigma
(y)\}\rangle _\CC$ is the linear span of $\sigma$-invariant set, i.e. a finite set with a label: if $x$ has $(a,b)$ as one of
its label, then we may take $(a+1,b)$ as the label.\end{proof}

\begin{remark}
The proof of Theorem \ref{g1} shows that we may omit the label $(0,g+1)$.
\end{remark}

\begin{proof}[Proof of Theorem \ref{g2}:]
By \cite[Theorem 1.1]{cov1} generic uniqueness holds for $\sigma _{g-1}(X(\CC))$, with the exceptions listed in the statement
of Theorem \ref{g2}.\end{proof}

\begin{proof}[Proof of Proposition \ref{g3}:]
Since $X(\CC)$ is a curve, its secant varieties are non-degenerate (\cite[Corollary 1.5]{a}, \cite[Remark 3.1 (i)]{cc1}). Thus
if
$r$ is odd we have $\sigma _{(r+1)/2}(X(\CC)) =\PP^r(\CC)$, while if $r$ is even we have $\sigma _{(r+2)/2}(X(\CC)) =
\PP^r(\CC)$ and $\sigma _{r/2}(X(\CC))$ is a hypersurface of  $\PP^r(\CC)$. Thus in the set-up of Theorem \ref{g1} we have $g =
\lfloor (r+2)/2\rfloor$. By \cite[Corollary 2.8]{cc2} generic uniqueness holds for  $\sigma _{\lfloor r/2\rfloor}(X(\CC))$.
Apply Theorem \ref{g1}.
\end{proof}

\begin{proof}[Proof of Theorem \ref{i4}:]
Use that $\sigma _k(X(\CC))$ is the join of $\sigma _{k-1}(X(\CC))$ and $X(\CC)$.
\end{proof}

\begin{proof}[Proof of Theorem \ref{e1}:] Let $W$ be a non-empty  Zariski open subset of $\sigma _{k-2}(X(\CC))$ such that for
each $q\in W$ there is a unique $S_q\subset X(\CC)$ with $|S_q|=k-2$ and $q\in \langle S_q\rangle$. The uniqueness of
$S_q$ gives
$\sigma (S_q)=S_q$ if $q\in W\cap \PP^r(\RR)$. Restricting if necessary $W$ we may assume that all $S_q$ are contained in
$X_{\reg}$. Thus
$S_q\cap X(\RR)=\emptyset$ for all $q$. Thus if $q$ has a label, then the label is of type $(\frac{k}{2}-1,0)$. Then we take the join
with $2$ copies of $X(\CC)$,i.e, to some $q\in W$ we associate all sets $S_q\cup \{y,z,\sigma (y),\sigma (z)\}\subset X(\CC)$
with $y, z\in X(\CC)\setminus X(\RR)$, $|\{y,z,\sigma (y),\sigma (z)\}|=4$ and $S_q\cap \{y,z,\sigma (y),\sigma
(z)\}=\emptyset$. Note that $S_q\cup \{y,z,\sigma (y),\sigma (z)\}$ is a $\sigma$-invariant set with label $(\frac{k}{2}+1,0)$.
\end{proof}

\end{document}